\documentclass[11pt]{amsart}
\usepackage{amsfonts,amssymb,amsmath,amsthm}
\usepackage{url}
\usepackage{enumerate}
\usepackage[all]{xy}

\usepackage{amsmath,amsfonts,amsthm,url,color,amssymb}
\usepackage{graphicx}
\usepackage{algpseudocode, algorithm}
\usepackage{hyperref}
\usepackage{todonotes}

\usepackage[norefs,nocites]{refcheck}

\newcommand{\G}{\Gamma}
\newcommand{\la}{\Lambda}
\newcommand{\C}{\mathcal C}
\newcommand{\Z}{\mathbb Z}
\newcommand{\ou}{\mathcal O}
\newcommand{\ord}{\mathrm{ord}}

\newcommand{\s}{\mathbb S}
\newcommand{\T}{\mathbb T}

\newtheorem{theorem}{Theorem}[section]
\newtheorem{lemma}[theorem]{Lemma}

\newtheorem{corollary}[theorem]{Corollary}
\newtheorem{proposition}[theorem]{Proposition}
\theoremstyle{definition}
\newtheorem{definition}[theorem]{Definition}
\newtheorem{example}[theorem]{Example}
\newtheorem{problem}[theorem]{Problem}
\theoremstyle{remark}
\usepackage{amssymb}
\newtheorem{remark}[theorem]{Remark}
\usepackage{enumerate}
\author{Arthur Fernandes}
\address{Departamento de Matem\'{a}tica, Universidade Federal de Minas Gerais, UFMG, Belo Horizonte, MG, 30270-901, Brazil.}

\email{arthurfpapa@gmail.com}
\curraddr{Universidade de São Paulo, Instituto de Ciências Matemáticas e de Computação,
São Carlos, SP 13566-590, Brazil}

\author{Claudio Qureshi}
\address{Instituto de Matemática y Estadística Rafael Laguardia, Facultad de Ingeniería, Universidad de la República, Montevideo 11300,
    Uruguay}
\email{cqureshi@gmail.com}

\author{Lucas Reis}
\address{Departamento de Matem\'{a}tica, Universidade Federal de Minas Gerais, UFMG, Belo Horizonte, MG, 30270-901, Brazil}
\curraddr{}
\email{lucasreismat@mat.ufmg.br}

\author{Sávio Ribas}
\address{Departamento de Matemática, Universidade Federal de Ouro Preto, Ouro Preto, MG, 35400-000,
   Brazil}
\curraddr{}
\email{savio.ribas@ufop.edu.br}

\begin{document}

\title{On a class of combinatorial group invariants}

\maketitle

\begin{abstract}
In this paper, we explore group invariants arising from combinatorial structures associated with finite groups, including the {\em functional graphs} of power maps and the well-studied {\em power graphs}. These invariants induce equivalence relations (and hence partitions) on the set of isomorphism classes of finite groups, which we classify from the finest to the coarsest. Surprisingly, all but three of these partitions turn out to coincide; for the subclass of nilpotent groups, all but two coincide. Furthermore, we introduce a broad class of nilpotent groups and show that, within this class, all but one of these partitions agree. Our proofs draw on tools and ideas from Combinatorics and Number Theory, while requiring only basic notions from Group Theory. In particular, we construct a general framework that may prove useful in contexts similar to those considered in this paper. Finally, we propose some open questions that emerge from our results.
\end{abstract}


\section{Introduction} 
\label{intro}

A key task in Group Theory is classifying groups up to isomorphism. Properties that stay the same under isomorphism, called {\em group invariants}, are important tools for this process. These invariants, such as order, exponent, abelianity and nilpotency, and the structure of its subgroups, capture key features that are independent of a particular representation of a group. Studying these invariants is important not only for distinguishing non-isomorphic groups~\cite{BroGa} but also for understanding related algebraic structures like group algebras~\cite{Gade}, where the invariants of the group affect the properties of the algebra. Most notably, group invariants are heavily used to treat the {\em Modular Isomorphism Problem}~\cite{Margo}, which explores the isomorphism of group algebras $kG$ with $G$ a finite $p$-group and $k$ a finite field of characteristic $p$.

In this paper we are interested in certain group invariants arising from combinatorial constructions, some of them being studied in past works. Our main goal is to compare these invariants through the equivalence relations they induce on the set of finite groups. Before moving to our results, we briefly introduce these invariants and position them in the literature.

\subsection*{The order multiset}
For a finite group $G$ and $g\in G$, let $\ord(g)$ be the order of $g$. 

\begin{definition}
    For a finite group $G$, we define $\ou(G)$ as the multiset $\{\mathrm{ord}(g) : g\in G\}$ comprising the order of the elements of $G$.
\end{definition}
 This invariant can give a lot of information on the group $G$, including whether the group is nilpotent~\cite{FeRe}. In~\cite{MDB}, groups $G, H$ with $\ou(G)=\ou(H)$ are called {\em conformal}. In~\cite{vas} groups $G, H$ with $\ou(G)=\ou(H)$ as simple sets (multiplicites not counted) are called {\em isospectral}. According to~\cite{McHa}, if $G$ and $H$ are abelian groups and $\ou(G)=\ou(H)$, then $G$ and $H$ are isomorphic. Also, according to~\cite{GS92}, if $G, H$ are groups of order $n$ with $\ou(G)=\ou(H)$, then the natural actions of the symmetric group $S_n$ on the cosets of $S_n/K$ with $K=G, H$ have the same permutation character. Actions with the same permutation character appear in Group Theory, Combinatorics, and Arithmetic Geometry, as many arguments depend only on conjugacy classes or counting fixed points, and not on the precise form of the action.
 More recently, further aspects of $\ou(G)$ are explored in~\cite{CaDe}.

\subsection*{The ``cyclic subgroup semi-lattice" of a group}
The set $S$ of all subgroups of a group $G$ has a poset structure with respect to the inclusion. In particular, every two elements in $S$ have an infimum (their intersection) and a supremum (the subgroup generated by them). In other words, $S$ is a {\em lattice}. This is the well studied {\em subgroup lattice} of a group $G$~\cite{Sch,Su,Susu}. Given a finite group $G$, let $\mathcal C_G$ be the set of its cyclic subgroups. Of course $\mathcal C_G$ is a poset with respect to the inclusion, so one could consider the ``cyclic sublattice". However, unless $G$ is cyclic, we only get a {\em semi-lattice} (every two elements have an infimum but not always have a supremum). It will be more convenient to give a digraph representation of this poset. 

\begin{definition}
    For a group $G$, let $\mathcal C_G$ be the set of its cyclic subgroups. We define $\la(G)$ as the {\em vertex-weighted} digraph whose vertices are elements of $\mathcal C_G$  (each weighted by the order of the corresponding cyclic group) and directed edges $C_1\to C_2$ if $C_2\subsetneq C_1$.
\end{definition}

We observe that $\la(G)$ is weakly connected (every element in $\mathcal C_G\setminus \{1_G\}$ is connected to $\{1_G\}$). Although we do not know any past work on this object, we refer to \cite{GaLi,SoZh,Tar1,TaLa} for papers exploring the proportion $\frac{\# \C_{G}}{\# G}$.

\subsection*{The power graph}
We have the following definition.

\begin{definition}
The {\em power graph} of a group $G$ is the graph ${\G}(G)$ with vertex set $G$ such that $g, h\in G$ are connected if $g\ne h$ and 
$g=h^k$ or $h=g^k$ for some integer $k$. 
\end{definition}
The power graph of (finite and infinite) groups has been extensively explored throughout the past years, from both algebraic and combinatorial points of view~\cite{Ca,CaGh,CGS,KeQu,KSCC,MiSc,MRS}. We also have a directed version $\vec{\Gamma}(G)$ of this graph, called the {\em power digraph} of $G$. In the latter the directed edges are $g\to g^k$ with $g\in G$ and $k\in \Z$ such that $g\ne g^k$. According to~\cite{Ca}, for a finite group $G$, the isomorphism class of $\vec{\G}(G)$ is uniquely determined by the one of $\G(G)$.  We recommend~\cite{KSCC} for a survey on power graphs and digraphs, and \cite{Ca,CaGh} for further connections between $\Gamma(G)$ and $\ou(G)$.

\subsection*{Digraphs of power maps}
Given a finite set $S$ and a function $f:S\to S$ we can associate to the pair $(S, f)$ a digraph with vertex set $S$ and edge set $s\to f(s)$. This is the {\em functional graph} of $f$ over $S$. 
The functional graph of many maps over finite algebraic structures have been considered in the past years~\cite{Lili,MPQ,PR18,QP18,Sh,Ug18,VS04}. When $S=G$ if a finite group, it is natural to consider the power maps $g\mapsto g^t$.

\begin{definition}
Given a finite group $G$ and $t\in \Z$, let $\G_t(G)$ be the directed graph with vertex set $G$ and edge set $\{g\to g^t\,:\, g\in G\}$.     
\end{definition}
The digraphs of power maps over finite groups have been explored in the past decade, with the main goal of describing $\G_t(G)$ explicitly. Recent works cover the cases where $G$ is cyclic~\cite{QuPa}, abelian~\cite{Bors,QuRe}, and also some classes of non-abelian groups~\cite{Ah,AhMo,DeZh}. In particular, for abelian groups, the digraphs $\G_t(G)$ present a remarkable symmetry: each connected component of $\G_t(G)$ comprises a cyclic graph whose vertices are the roots of pairwise isomorphic trees. This symmetry also arises in digraphs of special maps over other finite algebraic structures, many of which can be explained through a unified approach given in~\cite{QR19}. More recently, it was proved in~\cite{FeRe} that the graph $\G_t(G)$ exhibits this symmetry for every $t\in \Z$ if and only if the finite group $G$ is nilpotent.

It is direct to verify that the objects introduced above are group invariants (with the convention that for the case of graphs, equality means isomorphism). We have the following definition.

\begin{definition}
Let $G$ and $H$ be finite groups. We define the following relations.

\begin{enumerate}
    \item $G\sim_1 H$ if $G\cong H$ (the ordinary isomorphism of groups).
    \item $G\sim_2 H$ if there exists a bijection $\varphi:G\to H$ such that $\varphi(g^t)=\varphi(g)^t$ for every $g\in G$ and $t\in \Z$.
    \item $G\sim_3 H$ if $\la(G)\cong \la(H)$ as vertex-weighted digraphs.
    \item $G\sim_4 H$ if ${\G}(G)\cong {\G}(H)$.
    \item $G\sim_5 H$ if ${\G}_t(G)\cong {\G}_t(H)$ for every $t\in \Z$.
    \item $G\sim_6 H$ if $\ou(G)=\ou(H)$. 
\end{enumerate} 
\end{definition}

\begin{remark}\label{rem:P2-P5}
    We observe that the relations $\sim_j$ for $j=2,5$ are closely related. More precisely, the condition $G\sim_5 H$ asserts that, for each $t\in \Z$, there exists a bijection $f_t:G\to H$ such that $f_t(g^t)=f_t(g)^t$ for every $g\in G$. In contrast, $G\sim_2 H$ requires the existence of a single bijection $f:G\to H$ for which $f(g^t)=f(g)^t$ holds for all $g\in G$ and all $t\in \Z$, that is, one may take $f_t=f$ for every $t\in \Z$.
\end{remark}

We easily check that each $\sim_j$ is an equivalence relation, hence they induce partitions $P_j$ of the set of isomorphism classes of finite groups. Less trivially, those relations interact nicely with {\em direct products}: see Theorems~\ref{thm:direct} and~~\ref{thm:direct2} for more details. It is then natural to compare those relations. The relation $\sim_j$ is {\em finer} than $\sim_k$ if $G\sim_j H$ implies $G\sim_k H$, that is, if the partition $P_j$ is finer than $P_k$. In this case, we write $P_j\le P_k$. In particular, $\{P_i:1\le i\le 6\}$ has a poset structure. We easily check that each $\sim_j$ is invariant by isomorphisms, hence $P_1\le P_j$ for every $2\le j\le 6$. We have few known results, including the inequalities $P_1<P_4\le P_6$ (see~\cite{Ca,CaGh}) and $P_1<P_5\le P_6$ (see~\cite{FeRe}). Moreover, the validity of $P_4\ne P_5$ is explicitly posed as a problem in~\cite{lar}. With this notation in mind, our main result is the following theorem.

\begin{theorem}\label{thm:main}
For the class of all finite groups, we have
$$P_1<P_2=P_3=P_4\leq P_5<P_6.$$
\end{theorem}

If we restrict ourselves to nilpotent groups, we have the following result.
\begin{theorem}\label{thm:newmain}
For the class of finite nilpotent groups, we have
$$P_1<P_2=P_3=P_4=P_5<P_6.$$
\end{theorem}

We obtain the following corollary, which will be further used without mention. 

\begin{corollary}\label{cor:main}
   If $1\le j\le 6$ and $G\sim_j H$, then $\# G=\# H$. 
\end{corollary}
\begin{proof}
 If    $G\sim_j H$ for some $1\le j\le 6$, Theorem~\ref{thm:main} entails that $G\sim _6 H$, that is, $\ou(G)=\ou(H)$. However, for every finite group $K$, the multiset $\ou(K)$ has exactly $\# K$ elements and the result follows.
\end{proof}

Given Theorems~\ref{thm:main} and~\ref{thm:newmain}, it is natural to ask how the inequalities $P_1< P_2$ and $P_5< P_6$ are affected when one restricts to a special class of finite groups. For instance, according to~\cite{McHa}, we have $P_6\le P_1$ for the class of abelian groups, hence for these groups Theorem~\ref{thm:main} yields the equalities
$$P_1=P_2=P_3=P_4=P_5=P_6.$$

The proof of \( P_5 \ne P_6 \) in the Theorem~\ref{thm:main} is obtained through an explicit example, namely a pair of groups of order $16$ that are equivalent under \( \sim_6 \) but not under \( \sim_5 \). However, when trying to obtain further classes of examples, we found \( p \)-groups of order $p^{p+1}$ with more intricate structures: see Remark~\ref{rem:example} for more details. This phenomenon motivated a deeper investigation into \( p \)-groups, leading us to introduce the concept of \emph{$\alpha$-groups}. This is a class of finite groups $G$, closely related to abelian groups when one considers the power maps $g\mapsto g^{p^i}$ over the set of $p$-elements of $G$, for each prime divisor $p$ of $\# G$: see Definition~\ref{def:alpha}. In particular, this is a large class of finite groups, closed under direct products, that contains the well studied {\em regular $p$-groups}: see Section~\ref{sec:alpha} for more details. For the nilpotent groups in this class, we have the following result. 

\begin{theorem}\label{thm:main-alpha}
    For the class of nilpotent $\alpha$-groups, we have 
    $$P_1<P_2=P_6.$$
\end{theorem}

According to Theorem 7.1(b) in~\cite{GPO}, if \( p > 2 \) is a prime and \( G \) is a \( p \)-group of order \( \le p^p \), then \( G \) is a regular \( p \)-group. In particular, \( P_5 = P_6 \) for the class of \( p \)-groups of order \( \le p^p \). This partly explains why the \( p \)-groups we found with \( P_5 \ne P_6 \) had order \( p^{p+1} \) and exhibited more complex structures.

We end this section with some important notation and the structure of the paper. For a group $G$, we fix the following notation.

\begin{itemize}
    \item $\mathcal C_G$ denotes the set of all cyclic subgroups of $G$.
    \item For each prime $p$, $\mathcal C_G^{(p)}$ denotes the collection of cyclic subgroups of $G$ whose order equals $p^i$ for some $i\ge 0$.

    \item $1_G$ is the identity element of $G$.

    \item For $g\in G$, $\langle g\rangle$ denotes the cyclic subgroup of $G$ generated by $g$.

    \item For $g\in G$, $\ord(g)$ denotes the order of $g$.
\end{itemize}

The paper is organized as follows. In Section 2 we provide some background machinery, which might be of independent interest. In particular, we prove that the relations $\sim_j$ interact nicely with direct products of finite groups. In Section 3 we provide the proof of Theorem~\ref{thm:main}, which is divided in many steps. In Section 4 we prove Theorem \ref{thm:newmain}, concerning the nilpotent case. In Section 5 we introduce and study $\alpha$-groups, and provide the proof of Theorem~\ref{thm:main-alpha}. Finally, in Section 6 we provide concluding remarks and propose some open problems.

\section{Preliminaries}
In many proofs of results in this paper it will be relevant to study the structure of the cyclic subgroups of an arbitrary finite group. In this context, the following definition is crucial.

\begin{definition}
Let $G$ be a finite group and let $\mathcal{F}=\{C_1,C_2,\ldots,C_s\} \subseteq \mathcal{C}_G$ be a nonempty subfamily of cyclic subgroups of $G$. A {\em compatible system of generators for $\mathcal{F}$} is a subset $S=\{g_1,\ldots, g_s\}\subseteq G$ where each $g_i$ generates $C_i$ and $$ g_i^{\frac{\#C_i}{\#(C_i \cap C_j)}} =   g_j^{\frac{\#C_j}{\#(C_i \cap C_j)}} , 1\leq i ,j \leq s.$$ 
We say that $S$ is a {\em compatible system of generators for $G$} if it is a compatible system of generators for the set $\mathcal{C}_G$.
\end{definition}

\begin{example}
    The set $\{(1, 1), (0, 1), (1, 0), (0, 2), (0, 0)\}$ is a compatible system of generators for $\Z_2\times \Z_4$ (written additively).
\end{example}

The following definition is useful.

\begin{definition}
 A nonempty subset $S\subseteq \C_G$ is hereditary if for every $C\in S$ and every subgroup $C_0$ of $C$ we have $C_0\in S$.    
\end{definition}

We have the following result.

\begin{theorem}\label{lem:comp}
    Every finite group has a compatible system of generators.
\end{theorem}
\begin{proof}
The result is obvious if $\# G=1$ or $\# G$ is a prime number, so we can assume $\# \mathcal{C}_G >2$. Consider the set $\Theta$ of all hereditary non-empty subsets $\mathcal{F}\subseteq \mathcal{C}_G$ such that $\mathcal{F}$ admits a compatible system of generators $S_\mathcal{F}$. Clearly $\Theta$ is nonempty since $\{\{1_{G}\}\}\in \Theta$. Among of these sets we take $\mathcal{F}$ maximal (regarding inclusion) and we will prove that $\mathcal{F}=\mathcal{C}_G$. By the sake of contradiction, let us suppose that there exists $C \in \mathcal{C}_G$ such that $C \not\in \mathcal{F}$. Among of these cyclic subgroups we take $C$ minimal (regarding inclusion). The minimality of $C$ guarantees that the family $\mathcal{F}'=\mathcal{F}\cup \{C\}$ is also hereditary. To lead a contradiction it suffices to prove that $\mathcal{F}'$ admits a compatible system of generators. Note that $\#C>1$ since $\{1_G\} \in \mathcal{F}$ and $\#C$ is not a prime number, otherwise $\mathcal{F}'$ admits the compatible system of generators $S_{\mathcal{F}}\cup \{g\}$ where $g$ is any generator of $C$ and $\mathcal{F}'\in \Theta$, which contradicts the maximality of $\mathcal{F}$. Set $n=\# C$, let $C_1,C_2,\ldots,C_t$ be the maximal proper subgroups of $C$ and $n_i=\#C_i$. Observe that the numbers $p_i:=[C:C_i]=\frac{n}{n_i}, 1\le i\le t,$ are pairwise distinct primes. By the minimality of $C$ we have $C_i\in \mathcal{F}$ for $1\leq i \leq t$. Let $g_0$ be any generator of $C$ and $g_i\in S_{\mathcal{F}}$ be the distinguished generator of $C_i$ for $1\leq i \leq t$. Since $g_i$ and $g_0^{p_i}$ are two generators of $C_i$, there exist integers $h_i$ with $\gcd(h_i,n_i)=1$ such that $(g_0^{p_i})^{h_i}=g_0^{h_ip_i}=g_i$ for $1\leq i \leq t$. There are two cases to consider:

\begin{enumerate}[(i)]
    \item $t=1$. In this case, $\#C>p_1$, $p_1\mid n_1$ and $\gcd(h_1,n)=1$, thus $\gamma:=g_0^{h_1}$ is a generator of $C$ and ${\gamma}^{p_1}=g_1$. We claim that $S_{\mathcal{F}}\cup \{\gamma\}$ is a compatible system of generators for $\mathcal{F}'$. Indeed, let  $H \in \mathcal{F}$ and $h  \in S_{\mathcal{F}}$ be the distinguished generator of $H$. We observe that $H \cap C \ne C$ (because $\mathcal{F}$ is hereditary and $C\not\in \mathcal F$), hence $H\cap C \subseteq C_1$ . In particular, $H\cap C = H \cap C_1$ and we obtain:
\[ \gamma^{\frac{\# C}{\#(C \cap H)}} = (\gamma^{p_1})^{\frac{\# C_1}{\#(C_1\cap H)}} = g_1^{\frac{\# C_1}{\#(C_1\cap H)}}=h^{\frac{\# H}{\# (C_1 \cap H)}}=h^{\frac{\# H}{\# (C \cap H)}}.\]
This proves $\mathcal{F}'\in \Theta$, contradicting the maximality of $\mathcal{F}$.

\item $t>1$. In this case,  we consider the system of congruences: \begin{equation}\label{eq:chinese}
    X\equiv h_i \pmod{n_i}\,, 1\leq i \leq t.
\end{equation}
By the Chinese Remainder Theorem, this system has a solution modulo $n=\operatorname{lcm}(n_1,\ldots, n_t)$ if and only if $h_i \equiv h_j \pmod{\gcd(n_i,n_j)}$ for every $1\leq i <j\leq t$. Since 
$\frac{\# C_i}{\# (C_i \cap C_j)} = \frac{n/p_i}{n/p_ip_j}=\frac{p_ip_j}{p_i}=p_j$,
we obtain $g_i^{p_j}=g_j^{p_i}$. Therefore, $g_0^{h_ip_ip_j}=g_i^{p_j}=g_j^{p_i}=g_0^{h_jp_ip_j}$ and so $h_ip_ip_j\equiv h_j p_ip_j \pmod{n}$. The latter is equivalent to $h_i\equiv h_j \pmod{\frac{n}{p_ip_j}}$. Since the primes $p_i$ are distinct, we obtain $\gcd(n_i, n_j)=\frac{n}{p_ip_j}$ and so the conditions in the Chinese Remainder Theorem are fulfilled. Thus, the system of equations in~\eqref{eq:chinese} has a solution. From hypothesis, $\gcd(n_i, h_i)=1$. Therefore, fixing a solution $P$ to this system of equations, $P$ is relatively prime with $n=\operatorname{lcm}(n_1,\ldots, n_t)$.  
In particular, $\gamma:=g_0^{P}$ is a generator of $C$. Moreover, as $Pp_i\equiv h_ip_i\pmod {n}$ for every $1\le i\le t$, we obtain  
\begin{equation}\label{eq:comp}
    \gamma^{p_i}=g_0^{P\cdot p_i}=g_0^{h_ip_i}=g_i,\;  1\leq i \leq t. 
\end{equation}

Now given $H\in \mathcal{F}$, we argue as in the previous item and obtain $C\cap H\ne C$. Therefore, $C\cap H=C_i\cap H$ for some $1\le i\le t$. 
In this case, Eq.~\eqref{eq:comp} gives

 $$\gamma^{\frac{\# C}{\#(C \cap H)}}=(\gamma^{p_i})^{\frac{\# C_i}{\#(C_i \cap H)}} =(g_i)^{\frac{\# C_i}{\#(C_i \cap H)}}=h^{\frac{\# H}{\# (C_i \cap H)}}=h^{\frac{\# H}{\# (C \cap H)}}$$ where $h\in S_{\mathcal{F}}$ is the distinguished generator of $H$. Thus $\mathcal{F}'\in \Theta$, contradicting the maximality of $\mathcal{F}$. Therefore, we proved $\mathcal{F}=\mathcal{C}_G$, hence $\mathcal{C}_G$ admits a compatible system of generators.  
\end{enumerate}
\end{proof}

If $S$ is a compatible system of generators for $G$ and $\mathcal F\subseteq \C_G$ is nonempty, then each element $C\in \mathcal F$ is generated by some $g\in S$. Therefore, we directly obtain the following result.

\begin{corollary}\label{cor:comp}
For every finite group $G$ and each nonempty set $\mathcal F\subseteq \C_G$, there exists a compatible system of generators for $\mathcal F$.
\end{corollary}

\subsection{On power closed sets}\label{subsection:power-closed}

Many steps in the proof of Theorem~\ref{thm:main} rely on inductive arguments. In particular, it will be convenient to extend some of our combinatorial objects to special sets that are not groups.

\begin{definition}
   For each nonempty set $A\subseteq \C_{G}$, set $\cup A=\bigcup_{C\in A}C\subseteq G.$

\begin{enumerate}[(i)]
    \item  $\la(A)$ stands for the subgraph of $\la(G)$ induced by $A$;
    \item $\Gamma(A)$ (resp. $\Gamma_t(A)$) denotes the subgraph of $\Gamma(G)$ (resp. $\Gamma_t(G)$) induced by $\cup A$;
    \item $\ou(A)$ stands for the multiset 
    $$\{\ord(g): g\in \cup A\}.$$
\end{enumerate}

\end{definition}

We observe that if $A\subseteq \C_G$, then $\cup A$ is a {\em power closed} set, that is, $g^t\in A$ for every $g\in A$ and $t\in \Z$. It is not hard to see that the converse is also true: any power closed subset of $G$ equals $\cup A$ for some $A\subseteq \C_G$ (in fact, we can always take $A$ hereditary). 

\begin{remark}
    For $A=\C_G$, we have $\cup A=G$, so the combinatorial objects defined above agree with the ones we earlier defined for finite groups. 
\end{remark}

\subsection{Products of groups and graphs}
 For directed graphs $\mathcal G, \mathcal H$ with vertex sets $U$ and $V$, respectively, their {\em tensor product} is the direct graph $\mathcal G\otimes \mathcal H$ with vertex set $U\times V$ and directed edges $(u, v)\to (u', v')$ where $u\to u'$ is an edge of $\mathcal G$ and $v\to v'$ is an edge of $\mathcal H$. 

\begin{remark}\label{rem:tensor}
It follows by the definition of tensor product of graphs that, for groups $K, L$, we have $\Gamma(K\times L)\cong \Gamma_t(K)\otimes \Gamma_t(L)$ for every $t\in \Z$.
\end{remark}

As follows we provide two theorems on the equivalence relations $\sim_i$ through direct product of groups. We will make use of Theorem~\ref{thm:main} to shorten the proofs. However, we emphasize that the proof of Theorem~\ref{thm:main} is completely independent from these two results.

\begin{theorem}\label{thm:direct}
    Let $1\le j\le 6$ and let $G_1, G_2, H_1$ and $H_2$ be finite groups such that $G_i\sim_j H_i$ for $i=1, 2$. Then $G_1\times G_2\sim_j H_1\times H_2$.\end{theorem}
\begin{proof}
We split the proof into cases.
\begin{enumerate}[(i)]
    \item For $j=1$ the result is trivial.
    \item For $2\le j\le 4$, Theorem~\ref{thm:main} entails that it suffices to consider the case $j=2$. Assume $G_i\sim_2 H_i$ for $i=1, 2$ and let $\varphi_i:G_i\to H_i$ be a bijection such that $\varphi_i(g^t)=\varphi_i(g)^t$ for every $g\in G_i$ and $t\in \Z$. Hence, if we set $\varphi:G_1\times G_2\to H_1\times H_2$ with $\varphi(g_1, g_2)=(\varphi_1(g_1), \varphi_2(g_2))$, then $\varphi$ is a bijection and $\varphi(g^t)=\varphi(g)^t$ for every $g\in G_1\times G_2$ and $t\in \Z$. Hence $G_1\times G_2\sim_2 H_1\times H_2$.

    \item For $j=5$, suppose $G_i\sim_5 H_i$ for $i=1, 2$. In particular, for every $t\in \Z$, we have $\Gamma_t(G_i)\cong \Gamma_t(H_i)$. The latter combined with Remark~\ref{rem:tensor} yields 
    $$\Gamma_t(G_1\times G_2)\cong \Gamma_t(G_1)\otimes \Gamma_t(G_2)\cong \Gamma_t(H_1)\otimes \Gamma_t(H_2)\cong \Gamma_t(H_1\times H_2),$$
and so $G_1\times G_2\sim_5 H_1\times H_2$.

    \item Finally for $j=6$, we observe that if for a group $G$ we denote by $n(G, a)$ the number of elements $g\in G$ with order $a$, then $$n(G\times H, a)=\sum_{\mathrm{lcm}(b, c)=a}n(G, b)\cdot n(H, c).$$
In particular, if $\ou(G_i)=\ou(H_i)$, then $n(G_i, a)=n(H_i, a)$ for every $a\in \Z_{>0}$ and so $n(G_1\times G_2, a)=n(H_1\times H_2, a)$ for every $a\in \Z_{>0}$. The latter is equivalent to $\ou(G_1\times G_2)=\ou(H_1\times H_2)$.
\end{enumerate}
\end{proof}

\begin{corollary}
    If $G, H$ and $K$ are finite groups with $G\sim_j H$ for some $1\le j\le 6$, then $G\times K\sim_j H\times K$.
\end{corollary}

We also provide a weak converse of the previous theorem.

\begin{theorem}\label{thm:direct2}
Let $1\le j\le 6$ and let $G_1, G_2, H_1, H_2$ be finite groups such that $G_1\times G_2\sim_{j}H_1\times H_2$. If $\# G_i=\# H_i$ for $i=1, 2$ and $\gcd(\# G_1, \# G_2)=1$, then $G_i\sim_j H_i$ for $i=1, 2$.     
\end{theorem}
\begin{proof}
We split the proof into cases.

\begin{enumerate}[(i)]
    \item For $j=1$ just note that, as $\gcd(\# G_1, \# G_2)=1$, the groups $G_1\times\{1_{G_2}\}\cong G_1$ and $H_1\times\{1_{H_2}\}\cong H_1$ are the unique subgroups of $G_1\times G_2$ and $H_1\times H_2$, respectively, of order $\# G_1$. Hence $G_1\cong H_1$. Similarly we obtain $G_2\cong H_2$. 
    
    \item For $2\le j\le 4$, Theorem~\ref{thm:main} entails that it suffices to consider the case $j=3$. Suppose $G_1\times G_2\sim_{3}H_1\times H_2$ an let $F:\mathcal C_{G_1\times G_2}\to \C_{H_1\times H_2}$ be the map inducing the isomorphism between the vertex-weighted digraphs $\Lambda(G_1\times G_2)$ and $\Lambda(H_1\times H_2)$. Since $\gcd(\#G_1, \#G_2)=1$, we can identify $\C_{G_1\times G_2}$ with $\C_{G_1}\times \C_{G_2}$ and, in this case, the elements of $\C_{G_1\times G_2}$ with order relatively prime to $\# G_2$ correspond to the cyclic subgroups of the direct product $G_1\times \{1_{G_2}\}$. A similar result holds for $H_1\times H_2$. In particular, since $F$ preserves inclusions (the edges) and the order of cyclic subgroups (the weight of the vertices), it follows that $F$ restricts to an isomorphism between $\Lambda(G_1\times \{1_{G_2}\})$ and $\Lambda(H_1\times \{1_{H_2}\})$. The latter easily implies $\la(G_1)\cong \la(H_1)$. Similarly we obtain $\la(G_2)\cong \la(H_2)$, concluding the proof for this case.

    \item For the case $j=5$, set $m_i=\# G_i=\# H_i$ and let $t$ be an aribitrary integer. Since $\gcd(m_1, m_2)=1$, the Chinese Remainder Theorem entails the existence of $s=s(t)\in \Z$ such that $s\equiv t\pmod {m_1}$ and $s\equiv 1\pmod {m_2}$. Therefore, $\Gamma_s(K)\cong \Gamma_t(K)$ for $K=G_1, H_1$. Moreover, for $L=G_2, H_2$, the graph $\Gamma_s(L)\cong \Gamma_1(L)$ comprises $m_2$ loops. It is direct to verify that, for a directed graph $\mathcal G$ and a loop $J$, we have $\mathcal G\otimes J\cong \mathcal G$.
Hence $\Gamma_s(G_1)\otimes \Gamma_s(G_2)$ comprises $m_2$ copies of $\Gamma_t(G_1)$. Analogously,  $\Gamma_s(H_1)\otimes \Gamma_s(H_2)$ comprises $m_2$ copies of $\Gamma_t(H_1)$. From hypothesis and Remark~\ref{rem:tensor}, we have the isomorphisms
$$\Gamma_s(G_1)\otimes \Gamma_s(G_2)\cong \Gamma_s(G_1\otimes G_2)\cong \Gamma_s(H_1\otimes H_2)\cong \Gamma_s(H_1)\otimes \Gamma_s(H_2).$$
In conclusion, the graphs $\Gamma_t(G_1)$ and $\Gamma_t(H_1)$ are isomorphic. In a similar way we obtain $\Gamma_t(G_2)\cong \Gamma_t(H_2)$. Since $t$ is arbitrary, we proved $G_i\sim_5 H_i$ for $i=1, 2$.

    \item $j=6$:  For a group $K$ and an integer $a$, let $n(K, a)$ be the number of elements in $K$ of order $a$. If $G_1\times G_2\sim_6 H_1\times H_2$, then $n(G_1\times G_2, a)=n(H_1\times H_2, a)$ for every $a\ge 1$. Let $d$ be any divisor of $\# G_1$, hence $\gcd(d, \# G_2)=1$ and so $n(G_1\times G_2, d)=n(G_1, d)$. Similarly we obtain 
$n(H_1\times H_2, d)=n(H_1, d)$ and so $n(G_1, d)=n(H_1, d)$. Since every element of $\ou(G_1)$ (hence $\ou(H_1)$) is a divisor of $\# G_1$, we obtain $\ou(G_1)=\ou(H_1)$, that is, $G_1\sim_6 H_1$. In a similar way we obtain $G_2\sim_6 H_2$. 
\end{enumerate}    
\end{proof}

The following proposition roughly says that, if $G$ is nilpotent, then $\la(G)$ can factor as simpler digraphs through tensor products.

\begin{proposition}\label{prop:prime}
    For each finite nilpotent group $G$, the isomorphism class of 
    $\la(G)$ (as a weighted digraph) is completely determined by the ones of $\la(\C_G^{(p)})$ with $ p|\# G$.
\end{proposition}

\begin{proof}
For each $C\in \C_G$, let $C^{(p)}$ be the largest cyclic subgroup of order a power of $p$ contained in $C$. Since $G$ is nilpotent, we easily conclude that

\begin{equation}\label{eq:prod-cyc}
\#C=\prod_{p|\# G}\# C^{(p)}.
\end{equation}
Moreover, for $C, C_0\in \C_G$, we have $C\subseteq C_0$ if and only if $C^{(p)}\subseteq C_0^{(p)}$ for every $p|\# C$. Therefore, if $\overline{\mathcal G}$ denotes the graph obtained by adding a loop at each vertex of a loopless graph $\mathcal G$, we see that  
\begin{equation}\label{eq:tensor-pr}
    \overline{\la(G)}\cong \bigotimes_{p|\# G}\overline{\la(\C_G^{(p)})}.
\end{equation}
Moreover, from Eq.~\eqref{eq:prod-cyc}, these digraphs are also isomorphic as vertex weighted digraphs with the convention that we multiply the weight of the vertices when taking the tensor product $\bigotimes_{p|\# G}\overline{\la(\C_G^{(p)})}$. The proof of the proposition follows by Eq.~\eqref{eq:tensor-pr} and the fact that the isomorphism class of a loopless vertex-weighted digraph $\mathcal G$ is completely determined by the one of $\overline{\mathcal G}$, and vice-versa.    
\end{proof}

\section{Proof of Theorem~\ref{thm:main}}
For the sake of clarity and organization, we break the proof of Theorem~\ref{thm:main} into two lemmas. The inequalities $P_1<P_2$ and $P_5<P_6$ are fairly easy, while the equalities $P_2=P_3=P_4\leq P_5$ are not trivial at all. Whenever relevant, we briefly give insights on the proofs of some steps.

\begin{lemma}\label{lem:examp}
    $P_1<P_3$ and $P_5<P_6$. 
\end{lemma}
\begin{proof}
We prove the inequalities separately.

\begin{enumerate}[(i)]
    \item The inequality $P_1\le P_3$ is trivial, so we just need to verify $P_1\ne P_3$. For the latter we provide infinitely many pairs of non isomorphic groups that are equivalent under $\sim_3$. Fix $p$ an odd prime, set $\mathcal A _p=\Z_p^3$ and let $\mathcal B_p=(\Z_p\times \Z_p)\rtimes \Z_p$ be the unique non abelian group of order $p^3$ and exponent $p$. It is direct to verify that $\Lambda(\mathcal A_p)$ and $\Lambda(\mathcal B_p)$ are isomorphic (in fact, both graphs comprise a directed rooted tree with exactly $\frac{p^3-1}{p-1}=p^2+p+1$ vertices of weight $p$, directed to the root which has weight $1$). Hence $\mathcal A_p\sim_3 \mathcal B_p$. However, such groups are clearly not isomorphic, since only one is abelian. 

\item The inequality $P_5\le P_6$ is Theorem 1.3 in~\cite{FeRe}. To prove that the inequality is strict, it suffices to prove that $P_5\ne P_6$. Let $A=\Z_4\times \Z_4$ and $B=Q_8\times \Z_2$, where $Q_8=\langle a, b\,|\, a^4=1, a^2=b^2, bab^{-1}=a^{-1} \rangle$ is the {\em quaternion group}. It is direct to verify that $\ou(G)=\ou(H)$, that is, $G\sim_6 H$. However, a direct computation reveals that $\Gamma_2(G)$ and $\G_2(H)$ are not isomorphic, hence $G\not\sim_5 H$.
\end{enumerate}    
\end{proof}
As follows, we provide an infinite family of examples of groups that are equivalent under $\sim_6$ but not under $\sim_5$.

\begin{remark}\label{rem:example}
Although the proof of $P_5\ne P_6$ is concluded by a single example using groups of order $16$, we note that there are less trivial examples. Let $p$ be an odd prime, $\mathbb A_p=\Z_{p^2}^2\times\Z^{p-3}_p$ and $$\mathbb B_p:=\left\langle a_1, a_2, \ldots, a_{p-1}, b\right\rangle,$$ where $a_1^{p^2}=1, a_i^p=1$ for $2 \leq i \leq p-1, b^p=a_1^p$ and all generators commute except that $b^{-1} a_i b=a_i a_{i+1}$ when $1 \leq i<p-1$, and $b^{-1} a_{p-1} b=a_{p-1} a_1^{-p}$. We directly verify that $\mathbb A_p$ and $\mathbb B_p$ are of exponent $p^2$, and have $p^{p-1}-1$ elements order $p$. According to Exercise $2.4$ of \cite{IFPG}, $\mathbb B_p$ is also a group of order $p^{p+1}$ and so $\ou(\mathbb A_p)=\ou(\mathbb B_p)$. Moreover, according to the same exercise, $\mathbb B_p$ has exactly $p$ elements of the form $g^p$: these correspond to the vertices in $\Gamma_p(\mathbb B_p)$ with positive indegree. However, it is clear that $\mathbb A_p$ has $p^2$ elements of the form $h^p$, hence $\Gamma_p( \mathbb A_p)$ and $\Gamma_p(\mathbb B_p)$ cannot be isomorphic. In particular, $\mathbb A_p\not\sim_5 \mathbb B_p$.
\end{remark}

We proceed to the proof of $P_4= P_3$ and $P_3 = P_2\le P_5$. For $P_4= P_3$ the idea relies on the fact that a group $G$ can be partitioned by grouping the elements of $G$ that generate the same cyclic subgroup. This partition will correspond to {\em cliques} in $\G(G)$ (that is, sets of vertices that are pairwise connected). Roughly speaking, if we ``contract" these cliques into a single vertex, what we obtain is an undirected copy of $\la(G)$. We can recover the directed graph thanks to the main result in~\cite{Ca} which states that the isomorphism class of $\G(G)$ is uniquely determined by the isomorphism class of its directed version $\vec{\G}(G)$. Moreover, we can reverse this procedure and obtain $\G(G)$ from $\la(G)$. The proof of $P_3\le P_2$ relies on explicitly constructing a map $f:G\to H$ with $f(g^t)=f(g)^t$ from a given isomorphism between $\la(G)$ and $\la(H)$. The construction of such map is a bit technical and heavily relies on the use of compatible systems of generators for groups, earlier introduced in this paper.

\begin{lemma}\label{lem:lat}
    $P_2= P_3=P_4\le P_5$.
\end{lemma}
\begin{proof}
 We split the proof into four parts: $P_3\le P_2, P_2\le P_4, P_4\le P_3$ and $P_2\le P_5$.
\begin{enumerate}[(i)]

\item $P_3\le P_2$:

Let $G, H$ be two groups such that $G\sim_3 H$, that is,  $\la(G)\cong \la(H)$. Let $\psi:\C_{G}\to \C_{H}$ be any bijection inducing an isomorphism between $\la(G)$ and $\la(H)$ (as vertex-weighted digraphs).  Let $\{g_1, \ldots, g_s\}$ be a compatible system of generators for $G$ (ensured by Theorem~\ref{lem:comp}) and set $C_i=\langle g_i\rangle$. It is clear that the sets $\psi(C_i)$ are the cyclic subgroups of $H$. Moreover, since $\psi$ preserves the weight of the vertices, the orders of $\psi(C_i)$ and $C_i$ coincide. For each $1\le i\le s$, let $h_i$ be a generator of $\psi(C_i)$ in a way that $\{h_1, \ldots, h_s\}$ is a compatible system of generators for $H$ (ensured by Theorem~\ref{lem:comp}). We define $\Psi:G\to H$ by $$\Psi(g_i^t)=h_i^t,$$ 
for every $t\in \Z$ and every $1\le i\le s$. From construction, every element of $G$ is a power of at least one $g_i$. In particular, if we prove that $\Psi$ is a well defined bijection, it will satisfy $\Psi(g^t)=\Psi(g)^t$ for every $g\in G, t\in \Z$ and we obtain $G\sim_2 H$. Let $d_i$ be the order of $C_i$, hence both $g_i$ and $h_i$ have order $d_i$.
\begin{itemize}
    \item $\Psi$ is well defined: if $g_i^a=g_j^b$ for integers $a, b$ and $1\le i, j\le s$, then $g_i^a, g_j^b\in C_i\cap C_j$. Since the elements $g_i$ form a compatible system of generators for $G$ there exists $g_k\in \{g_1, \ldots, g_s\}$, a generator for $C_k:=C_i\cap C_j$, such that 
     $$g_i^{d_i/d_{k}}=g_k=g_j^{d_j/d_k}.$$
Since $g_i^a=g_j^b$ is an element of $C_k$, we have $g_i^a=g_k^l=g_j^b$ for some $l\in \Z$ and then 
$$g_i^a=g_i^{ld_i/d_k}=g_j^{ld_j/d_k}=g_j^b.$$
Since $(g_i, h_i)$ and $(g_j, h_j)$ are pairs of elements with the same order, the last equalities entail that $h_i^a=h_i^{ld_i/d_k}$ and $h_j^b=h_j^{ld_j/d_k}$. Since $\psi$ is an isomorphism between $\la(G)$ and $\la(H)$, we must have $\psi(C_i\cap C_j)=\psi(C_i)\cap \psi(C_j)$. Moreover, the elements $h_i$ comprise a compatible system of generators for $h$, hence $h_i^{d_i/d_k}=h_j^{d_j/d_k}$ and then $h_i^{ld_i/d_k}=h_i^{ld_j/d_k}$. In conclusion, $$\Psi(g_i^a)=h_i^a=h_j^b=\Psi(g_j^b),$$ and so $\Psi$ is well defined.

\item $\Psi$ is a bijection: we have proved that $\Psi$ is well defined, and we observe that every element of $H$ is a power of some $h_i$, hence $\Psi$ is surjective. Since $G\sim_3H$, we obtain $\# G=\# H$. Therefore, $\Psi$ is a surjective map between two finite sets of the same cardinality and so $\Psi$ is also bijective. 
\end{itemize}

\item $P_2\le P_4$:

Let $G, H$ be groups with $G\sim_2 H$ and let $\varphi:G\to H$ be a bijection such that $\varphi(g^t)=\varphi(g)^t$ for every $g\in G$ and $t\in \Z$. The latter implies that, for elements $g, h\in G$, we have $g=h^m$ for some $m\in \Z$ if and only if $\varphi(g)=\varphi(h)^m$. In particular, we see that the following are equivalent:
\begin{itemize}
    \item $g, h\in G$ are connected in $\Gamma(G)$;
    \item $g\ne h$ and $g=h^k$ or $h=g^k$ for some $k\in \Z$;
    \item $\varphi(g)\ne \varphi(h)$ and $\varphi(g)=\varphi(h)^k$ or $\varphi(h)=\varphi(g)^k$ for some $k\in \Z$;
    \item $\varphi(g), \varphi(h)\in H$ are connected in $\Gamma(H)$.
\end{itemize}
Thus $\varphi$ is an isomorphism between $\Gamma(G)$ and $\Gamma(H)$, that is, $G\sim_4 H$.

    

\item $P_4\le P_3$:

Let $G, H$ be groups such that $G\sim_4 H$, that is, $\G(G)\cong \G(H)$. From the main result in~\cite{Ca}, the directed power graphs $\vec{\G}(G)$ and $\vec{\G}(H)$ are also isomorphic. In other words, there exists a bijection $f:G\to H$ such that, for every $g, g_0\in G$, we have $g=g_0^k\ne g_0$ for some $k\in \Z$ if and only if $f(g)=f(g_0)^l\ne f(g_0)$ for some $l\in \Z$. In particular, for each cyclic subgroup $C$ of $G$, its image $f(C)$ is a cyclic subgroup of $H$ and of the same order. 

For each $g\in G$, let $[g]$ be the set of all elements $g'\in G$ such that $g$ and $g'$ are powers of each other. It is clear that $[g]$ comprises the set of $g'\in G$ with $\langle g'\rangle=\langle g\rangle $. 
For each $h\in H$, define $[h]$ in a similar way. Let $C_1, \ldots, C_s$ be the cyclic subgroups of $G$ and let $g_i$ be a generator of $C_i$. Since $f$ is bijective and preserves adjacency, we easily obtain 
$f([g_i])=[f(g_i)]$, $f(g_i)$ generates $f(C_i)$ and $f([g_i])=f([g_j])$ if and only if $[g_i]=[g_j]$. In particular, since $f$ is surjective, the groups $D_i:=f(C_i)$ comprise the cyclic subgroups of $H$. Therefore, the map $\widetilde{f}:\C_{G}\to \C_{H}$ with $\widetilde{f}(C_i)=D_i$ is a bijection. We claim that it induces an isomorphism between $\la(G)$ and $\la(H)$. We have seen that $\widetilde{f}$ preserves the weight of the vertices. Moreover, we observe that the following are equivalent:
\begin{itemize}
    \item $C_i\to C_j$ in $\la(G)$;
\item $C_j\subsetneq C_i$;
\item $g_j$ is a power of $g_i$ but the converse is not true;
\item $f(g_j)$ is a power of $f(g_i)$ but the converse is not true;
\item $D_j\subsetneq D_i$;
\item $D_i\to D_j$ in $\la(H)$.
\end{itemize}
In other words, $\widetilde{f}$ preserves adjacency. Hence $\la(G)\cong \la(H)$ and so $G\sim_3 H$.



\item $P_2\le P_5$: 

This follows directly by Remark~\ref{rem:P2-P5}.

\end{enumerate}
\end{proof}

\section{Proof of Theorem \ref{thm:newmain}}
In this subsection we prove $G\sim_5 H$ implies $G\sim_3 H$ for the class of nilpotent groups. We first give some insights about the proof. We observe that $\la(G)$ basically describes the structure of the cyclic subgroups of $G$ with respect to the inclusion. This is somehow captured by suitable digraphs $\G_t(G)$: if $t$ is a prime dividing the order of $g\in G$, then the edge $g\to g^t$ in $\G_t(G)$ induces the edge $\langle g\rangle \to \langle g^t\rangle$ on $\la(G)$. 
The main difficulty lies in the fact that we cannot trivially guarantee the existence of a bijection $f:G\to H$ that keeps this structure as $t$ varies. The $p$-groups are  perfect prototypes for this approach since, in this case, we only need $t=p$. However, thanks to Proposition~\ref{prop:prime}, we can manage to adapt this idea for arbitrary finite nilpotent groups. Throughout this subsection, $G, H$ are groups such that $G\sim_5 H$ and $p$ is a prime dividing $\# G$ (hence $\# H$). We start with the following result.

\begin{lemma}
Let $G, H$ be finite groups with $G\sim_5 H$ and let $p$ be a prime divisor of $\# G=\# H$. Then $\G_p(\C_{G}^{(p)})\cong \Gamma_p(\C_H^{(p)})$.
\end{lemma}
\begin{proof}
For each finite group $K$ and each $g\in K$, let $K_g$ be the connected component of $g$ in $\G_p(K)$. 
Observe that, for $K=G, H$, we have $\G_p(\C_K^{(p)})=K_{1_K}$. Moreover, this digraph is a rooted tree with a loop at the root.
It suffices to prove that $G_{1_G}\cong H_{1_H}$. 
It will be desirable to obtain an isomorphism between the digraphs $\G_p(G)$ and $\G_p(H)$ mapping $1_G$ to $1_H$. However, we cannot obtain this directly. 

Suppose $f:G\to H$ is an isomorphism between $\Gamma_p(G)$ and $\Gamma_p(H)$ and set $h=f(1_G)$. Since $1_G$ corresponds to a loop in $\Gamma_p(G)$, the same must hold for $h$ in $\Gamma_p(H)$. Therefore, $h^p=h$ and then $h^{p-1}=1$.

{\bf Claim.} {\em The map $z\mapsto z^{p-1}$ induces an injective homomorphism from $H_{h}$ to $H_{1_H}$. }

{\em Proof of the Claim.} If $z_1\to z_2$ is an edge of $H_{h}$, then $z_2=z_1^p$ and there exist integers $i, j$ such that $z_1^{p^i}=z_2^{p^j}=h$. Hence $z_2^{p-1}=(z_1^{p-1})^p$ and 
$$(z_1^{p-1})^{p^i}=(z_2^{p-1})^{p^j}=h^{p-1}=1_H.$$
In other words, $z_1^{p-1}\to z_2^{p-1}$ is an edge of $H_{1_H}$. Now we prove that this map restricts to an injection. Suppose by contradiction that $h_1, h_2$ are distinct vertices of $H_h$ with $h_1^{p-1}=h_2^{p-1}$ and take such elements $h_i$ in a way that the distance from $h_1$ to the root $h$ of the tree $H_h$ is minimal. Therefore $(h_1^p)^{p-1}=(h_2^p)^{p-1}$. However, $h_1^p, h_2^p$ are vertices of $H_h$ with $h_1^p$ closer to $h$ than $h_1$.  From the minimality of $h_1$, such vertices cannot be distinct and then $h_1^p=h_2^p$. Since $h_1^{p-1}=h_2^{p-1}$ we obtain $h_1=h_2$, a contradiction. This proves the claim. \qed

The claim entails that $H_{1_H}$ contains a copy of $H_h$. As $H_h\cong G_{1_G}$, the graph $H_{1_H}$ contains a copy of $G_{1_G}$. We can argue similarly (taking $g=f^{-1}(1_H)$) and conclude that 
$G_{1_G}$ contains a copy of $H_{1_H}$. Since these digraphs are finite we obtain $G_{1_G}\cong H_{1_H}$, concluding the proof.
  
\end{proof}

The next step is to prove the following result.

\begin{center}
   {\em  If $G, H$ are finite groups with $\Gamma_p(\C_G^{(p)})\cong \Gamma_p(\C_H^{(p)})$, then $\la(\C_G^{(p)})\cong \la(\C_H^{(p)})$.}
\end{center}

The idea to prove the latter is to show that there is an isomorphism between $\Gamma_p(\C_G^{(p)})$ and $\Gamma_p(\C_H^{(p)})$ mapping a fixed compatible system of generators for $\C_G^{(p)}$ into a compatible system of generators for $\C_H^{(p)}$. If we guarantee this, $\la(\C_G^{(p)})$ can be viewed as a subgraph of $\la(\C_H^{(p)})$. Furthermore, the symmetric assumptions on $G, H$ provide the converse of the latter, concluding the desired isomorphism since these digraphs are finite. 

 For the sake of clarity and organization,  we first introduce some notation that we use only in the proof of the result above. 
 Recall that  $\Gamma_p(\C_G^{(p)})$ is always a rooted tree with a loop at the root corresponding to the vertex $1_G$. For each $g\in \cup \C_G^{(p)}$, let $T_g$ be the graph induced by the set of $\widetilde{g}\in G$ that are connected to $g$ by a directed path. Observe that $T_g$ is just the largest rooted subtree of $\Gamma_p(\C_G^{(p)})$ having $g$ as a root.
 For $g, \bar{g}\in G$ we write $g\approx \bar{g}$ if $T_g\cong T_{\bar{g}}$ and $g^p=\bar{g}^p$. It is clear that $\approx$ is an equivalence relation.
 
\begin{remark}\label{rem:approx}
    Observe that if $g\in \cup\C_G^{(p)}$, then $\langle g\rangle=\langle\bar{g}\rangle$ if and only if $\bar{g}=g^k$ for some $0<k\le \ord(g)$ with $\gcd(k, p)=1$. In this case, $\bar{g}\in \cup\C_G^{(p)}$ and the map $z\mapsto z^k$ induces an automorphism of $\Gamma_p(\C_G^{(p)})$ mapping $T_g$ to $T_{\bar{g}}$, hence $T_g\cong T_{\bar{g}}$. Moreover, if $g^p=\bar{g}^p$ and $\ord(g)=p^d$ with $d\ge 1$, we obtain $$g^p=\bar{g}^p=g^{pk}\iff g^{p(k-1)}=1\iff k\equiv 1\pmod {p^{d-1}}.$$ As $0< k<p^d$, we have exactly $N$ elements satisfying this property, where $N=p-1$ for $d=1$ and $N=p$ for $d>1$. 
 Therefore, each $g\in \cup\C_G^{(p)}$ of order $p^d>1$ yields exactly $N(g)$ elements $\bar{g}\in \cup\C_G^{(p)}$ such that $\langle g\rangle=\langle\bar{g}\rangle$ and $g\approx \bar{g}$, where $N(g)=p-1$ if $d=1$ and $N(g)=p$ if $d>1$.
\end{remark}

Now, fix $\{g_1, \ldots, g_s\}$ a compatible set of generators for $\C_G^{(p)}$ (ensured by Corollary~\ref{cor:comp}). For each $f:\cup\C_G^{(p)}\to \cup\C_H^{(p)}$ inducing an isomorphism between $\G_p(\C_G^{(p)})$ and $\G_p(\C_H^{(p)})$, a {\em bad pair} for $f$ is a pair $(g_i, g_j)$ with $i\ne j$ such that $f(g_i)$ and $f(g_j)$ generate the same cyclic subgroup. We observe that any isomorphism between $\G_p(\C_G^{(p)})$ and $\G_p(\C_H^{(p)})$ must preserve the roots, hence $f(1_G)=1_H$. In particular, since the order of an element $g\in \cup\C_G^{(p)}$ is $p^d$ if and only if the distance from $g$ to the root $1_G$ in $\G_p(\C_G^{(p)})$ is $d$, it follows that any isomorphism preserves the orders of the elements. In particular, if $(g_i, g_j)$ is a bad pair, then both elements have order $p^d$ for some $d\ge 1$. We can then associate to $f$ the sum
$$\mathcal E(f):=\sum_{1\le i<j\le s} d_f(g_i, g_j),$$
where $d_f(g_i, g_j)=0$ if $(g_i, g_j)$ is not a bad pair for $f$ and $d(g_i, g_j)=\# G^{-2d}$ if $(g_i, g_j)$ is a bad pair for $f$ comprising elements of order $p^d$.  In particular, $\mathcal E(f)\ge 0$ with equality if and only if $f$ does not have bad pairs.

We observe that if $f$ has a bad pair, then the set $\{f(g_i): 1\le i\le s\}$ cannot be contained in a compatible system of generators for $H$. The following lemma guarantees the existence of an isomorphism without bad pairs. The informal idea is to perform a sequence of ``flips" on suitable subtrees of the graphs $\G_p(\C_K^{(p)}), K=G, H$, in a way that we decrease the number of bad pairs of ``small" order for a fixed isomorphism. By construction, the value of $\mathcal E$ will decrease, hence  it will eventually vanish since we have a finite number of isomorphisms between $\G_p(\C_G^{(p)})$ and $\G_p(\C_H^{(p)})$. 

 \begin{lemma}
Let $G, H$ be finite groups of order divisible by $p$ with $\G_p(\C_G^{(p)})\cong \G_p(\C_H^{(p)})$.   Then there  exists an isomorphism  $f:\cup\C_G^{(p)}\to \cup\C_H^{(p)}$ between these two digraphs with $\mathcal E(f)=0$, i.e., without bad pairs.
 \end{lemma}
 \begin{proof}
Suppose by contradiction that every isomorphism  between $\G_p(\C_G^{(p)})$ and $\G_p(\C_G^{(p)})$ yields a positive value of $\mathcal E$. As there exist only a finite number of such isomorphisms, there exist one of them, say $F$, such that $\mathcal E(F)$ is minimal. Without loss of generality, assume $(g_1, g_2)$ is a bad pair for $F$ with smallest possible order $p^d$. For each $1\le i\le s$, set $h_i=F(g_i)$. From hypothesis, $h_1$ and $h_2$ generate the same cyclic subgroup. The same must hold for $h_1^p=F(g_1)^p=F(g_1^p)$ and $h_2^p=F(g_2)^p=F(g_2^p)$. However, since the elements $g_i$ comprise a compatible system of generators for $\C_G^{(p)}$, both $g_1^p, g_2^p$ are elements in $\{g_1, \ldots, g_s\}$. Therefore, $(g_1^p, g_2^p)$ would be a bad pair for $F$ of order $p^{d-1}<p^d$ unless $g_2^p=\gamma:=g_1^p$. Since $(g_1, g_2)$ is a bad pair for $F$ with smallest order, the last equality must occur and so 
$$\eta:=h_1^p=F(g_1)^p=F(g_1^p)=F(\gamma)=F(g_2^p)=F(g_2)^p=h_2^p.$$ 

As $h_1, h_2$ generate the same cyclic group, we conclude that $h_1\approx h_2$. It is clear that the relation $\approx$ is invariant under isomorphisms. Hence $g_1\approx g_2$. Assume, without loss of generality, that $\{g_1, \ldots, g_m\}$ with $m\ge 2$ is the subset of $\{g_1, \ldots, g_s\}$ comprising the elements $g_i$ such that $g_i\approx g_1$.  Hence for $1\le i\le m$, we have $h_i\approx h_1$. Moreover, the elements $g_i$ and $h_i$ with $1\le i\le m$ have all order $p^{d}$. 

The elements $g_i$ clearly do not generate the same cyclic subgroups (they belong to a compatible system of generators), hence Remark~\ref{rem:approx} entails that they produce $mN(g_1)$ elements $x\in G$ with $x\approx g_1$. As $F$ is an isomorphism,  there exist at least $mN(g_1)=mN(h_1)$ elements $y\in H$ with $y\approx F(g_1)=h_1$. Since $h_1$ and $h_2$ generate the same cyclic subgroup, Remark~\ref{rem:approx} entails that the contribution of $\langle h_1\rangle \cup\cdots \cup \langle h_m\rangle$ to such elements $y$ is at most $$(m-1)N(h_1)<mN(h_1).$$ Hence there exists $h'\in \cup\C_H^{(p)}$ with $h'\approx h_1$ in a way that $h'\not\in\langle h_i\rangle$ for $1\le i\le m$. In particular, $h'\ne h_i$ for $1\le i\le m$. If $h'=h_i$ for some $i>m$, we obtain $$g_i=F^{-1}(h_i)=F^{-1}(h')\approx F^{-1}(h_1)=g_1,$$ 
a contradiction with $i>m$.

Therefore, $g':=F^{-1}(h')\ne g_i$ for $1\le i\le s$. We clearly have the isomorphism 
 $$T_{h'}\cong T_{h_1}.$$
Moreover, the vertices $h', h_1$ are directed to the same vertex $\eta=h_1^p$. In particular, there exists a ``flip" on $\Gamma_p(\C_H^{(p)})$ that switches those trees. In other words, there exists an automorphism $\Theta$ of $\Gamma_p(\C_H^{(p)})$ such that $\Theta(T_{h'})=T_{h_1}, \Theta(T_{h_1})=T_{h'}$ and $\Theta$ fixes any other vertex out of these two trees. It is clear that $\Psi:=\Theta\circ F$ is an isomorphism between $\Gamma_p(\C_G^{(p)})$ and $\G_p(\C_H^{(p)})$. Moreover, observe that the vertices in $T_{h'}, T_{h_1}$ correspond to elements of order $\ge  p^{d+1}$ and the elements $h_1, h'$ which have order $p^d$. Hence $\Psi(g)=F(g)$ whenever $\ord(g)\le p^{d}$ with the exception of $\Psi(g_1)=\Theta(h_1)=h'$ and $\Psi(g')=\Theta(h')=h_1$.

We claim that $\mathcal E(\Psi)<\mathcal E(F)$. In fact, let $(g_a, g_b)$ be a bad pair for $\Psi$ of order $\le p^d$. Recall that $g'$ is not equal to any of the elements $g_i$. Therefore, if $1\not\in \{a, b\}$, then $\Psi(g_k)=F(g_k)$ for $k\in \{a, b\}$ and so $(g_a, g_b)$ is also a bad pair for $F$. Now we note that no pair $(g_1, g_b)$  can be bad for $\Psi$. In fact, assume $(g_1, g_b)$ is a bad pair for $\Psi$, hence $g_b$ has order $p^d$. Since the elements $g_i$ form a compatible system of generators for $\C_G^{(p)}$, both $g_1^p$ and $g_b^p$ are in $\{g_1, \ldots, g_s\}$, but none of them are equal to $g_1$. Therefore if $g_1^p\ne g_b^p$, then $(g_1^p, g_b^p)=(g_k, g_l)$ are elements of order $p^{d-1}< p^d$ with $1\not\in \{k, l\}$.  From previous observation, $(g_1^p, g_b^p)$ is also a bad pair for $F$. The latter contradicts our assumption that $(g_1, g_2)$ is a bad pair for $F$ of smallest order. Thus $g_1^p=g_b^p$ and so 

$$(h')^p=\Psi(g_1)^p=\Psi(g_1^p)=
\Psi(g_b^p)= 
\Psi(g_b)^p=F(g_b)^p=h_b^p.$$
In other words, $h_b^p=(h')^p=\eta$. Since $h_b, h'$ generate the same subgroup, Remark~\ref{rem:approx} entails that $h_b\approx h'$. Moreover, as $h'\approx h_1$, we obtain $h_1\approx h_b$. Therefore, $g_b=F^{-1}(h_b)\approx F^{-1}(h_1)=g_1$. However, from our assumption on the elements $g_i$, the latter implies $1\le b\le m$. This is a contradiction with the fact that $h'$ is not in the cyclic group generated by any of the elements $h_i$ with $1\le i\le m$. In conclusion, the following hold:

\begin{itemize}
    \item every pair $(g_a, g_b)$ of order $\le p^d$ with $1\not\in \{a, b\}$ that is bad for $\Psi$ is also bad for $F$;
    \item no pair $(g_1, g_b)$ is bad for $\Psi$;
    \item the pair $(g_1, g_2)$ is bad for $F$.
\end{itemize}
The latter implies 
$$\mathcal E(F)-\mathcal E(\Psi)\ge d_F(g_1, g_2)-\Delta=\# G^{-2d}-\Delta,$$
where $\Delta$ is the sum of $d_{\Psi}(g_a, g_b)$ running over all the pairs of order $\ge p^{d+1}$ with $a<b$. Of course we have $d_{\Psi}(g_a, g_b)\le \# G^{-2(d+1)}$ for any such pair, and there are at most $\frac{\# G^2}{2}$ of them. In particular, we have $\Delta\le \frac{\#G^2}{2}\cdot \# G^{-2(d+1)}$ and so 
$$\mathcal E(F)-\mathcal E(\Psi)\ge \# G^{-2d}-\frac{\# G^{-2d}}{2}=\frac{\# G^{-2d}}{2}>0.$$
Thus $\mathcal E(F)>\mathcal E(\Psi)$, a contradiction with the minimality of $\mathcal E(F)$. 
\end{proof}

The previous lemma gives us the following corollary.

\begin{corollary}
Let $G, H$ be groups with $\Gamma_p(\C_G^{(p)})\cong \G_p(\C_H^{(p)})$. Then $\la(\C_G^{(p)})\cong \la(\C_H^{(p)})$.
\end{corollary}

\begin{proof}
Let $\{g_1, \ldots, g_s\}$ be a compatible system of generators for $\C_G^{(p)}$ (ensured by Corollary~\ref{cor:comp}) and let $C_i$ be the cyclic subgroup generated by $g_i$. 
From the previous lemma, there exists a bijection $F:\cup\C_G^{(p)}\to \cup\C_H^{(p)}$ such that $F(g_i^p)=F(g_i)^p$ and the elements $h_i=F(g_i)$ generate pairwise distinct cyclic subgroups of $H$. Let $D_i$ be the cyclic subgroup of $H$ generated by $h_i$. It is clear that the groups $C_i$ are the cyclic subgroups of $G$. We set $\mathcal F:\C_{G}^{(p)}\to \C_H^{(p)}$ with $\mathcal F(C_i)=D_i$. From hypothesis,  $\mathcal F$ is injective. Moreover, as $F$ preserves the order of the elements, $\# D_i=\#\mathcal F(C_i)=\# C_i$. Finally, we prove that $\mathcal F$ is also a graph homomorphism. If $C_i\to C_j$ is an edge of $\la(\C_G^{(p)})$, then $C_j\subsetneq C_i$. Since the elements $g_i$ are in a compatible system of generators for $\C_G^{(p)}$, the latter implies $g_j=g_i^{p^s}$ for some $s\ge 1$ and so 
$$h_j=F(g_j)=F(g_i^{p^s})=F(g_i)^{p^s}=h_i^{p^s}.$$
Thus $D_j\subsetneq D_i$, that is, $D_j\to D_i$ is an edge in $\la(H)$. In conclusion, $\mathcal F$ is an injective graph homomorphism that preserves the weights of the vertices. Therefore, we can see $\la(\C_G^{(p)})$ as a subgraph of $\la(\C_H^{(p)})$. However, the converse must also hold (switch $G$ and $H$). Since these digraphs are finite, we obtain $\la(\C_H^{(p)})\cong \la(\C_G^{(p)})$.
\end{proof}

The proof of $P_5\le P_3$ for the class of finite groups follows directly by the previous corollary and Proposition~\ref{prop:prime}. This concludes the proof of Theorem~\ref{thm:newmain}.

\section{Introducing $\alpha$-groups}\label{sec:alpha}
It is well known that, for a finite abelian group $G$ and $g\in G$, if the equation $x^n=g$ has a solution, then the number of solutions for such equation equals the number of solutions for the equation $x^n=1_G$. The following definition seeks to explore this property for more general groups by restricting it to special pairs $(n, g)$.

\begin{definition}\label{def:alpha}
    Let $G$ be a finite group. A nonempty subset $S\subseteq \C_G$ is an $\alpha$-set if $S$ is hereditary and, in addition, there exists a prime divisor $p$ of $\# G$ such that:
    
    \begin{enumerate}[(i)]
        \item $S\subseteq \C_G^{(p)}$;
        \item for each $i\ge 1$ and every $g\in \cup S$, the number of solutions to $x^{p^i}=g$ in $\cup S$ is either $0$ or 
    $$n_i(S):=\#\{x\in \cup S\colon x^{p^i}=1_G \}.$$
    \end{enumerate}
Moreover, $G$ is an $\alpha$-group if $\C_G^{(p)}$ is an $\alpha$-set for every prime divisor $p$ of $\# G$.    
\end{definition}
\begin{example}\label{ex:alpha}
Let $p$ be a prime number and let $G$ be a finite group of exponent $p$, hence $g^p=1_G$ for every $g\in G$. Therefore, if $g\ne 1_G$ and $i\ge 1$, the equation $x^{p^i}=g$ has no solution in $G$. In particular, $G$ is an $\alpha$-group. 
\end{example}

Abelian groups are clearly $\alpha$-groups. More generally, it is direct to verify that the following result holds.

\begin{lemma}\label{lem:alpha-ex}
Let $G$ be a finite group such that, for every prime $p$ dividing $\# G$, the following holds: if $x, y\in \cup \C_G^{(p)}$, then for every $i\ge 1$ we have $x^{p^i}=y^{p^i}$ if and only if $(xy^{-1})^{p^i}=1_G$. Then $G$ is an $\alpha$-group. 
\end{lemma}
 
A direct consequence of Lemma~\ref{lem:alpha-ex} is that $\alpha$-groups include all the {\em regular $p$-groups}: see Theorem 7.2 in~\cite{GPO}. Recall that a finite $p$-group $G$ is a regular $p$-group if, for every $x,y\in G$, we have $$x^p\cdot y^p=(xy)^p\cdot c^p,$$ for some $c$ in the derived subgroup $\langle x,y\rangle'$ of the group generated by $x$ and $y$. In particular, abelian $p$-groups are regular $p$-groups: just take $c=1$. Moreover, the converse is also true for $p=2$: see Theorem~2.8 in~\cite{IFPG}. However, for every $p>2$, there are many regular $p$-groups that are not abelian. Regular $p$-groups play an important role in Group Theory and they serve as a natural generalization of abelian $p$-groups. For more information on regular $p$-groups, see~\cite{GPO} and ~\cite{IFPG}.

\begin{remark}
The definition of $\alpha$-groups is primarily motivated by $p$-group constructions, but it also includes certain groups that are not nilpotent. For instance, when $n$ is odd, it is direct to check that the dihedral group $D_{2n}$ is an $\alpha$-group.
\end{remark}

In the following lemma we prove that $\alpha$-groups are closed under direct products. 

\begin{lemma}
If $G_1, G_2$ are $\alpha$-groups, then so is $G_1\times G_2$. 
\end{lemma}
\begin{proof}
Let $p$ be a prime dividing both $\# G_1$ and $\# G_2$. It is direct to verify that any $p$-element in $G_1\times G_2$  is of the form $g=(g_1, g_2)$ with $g_j\in \cup \mathcal C_{G_j}^{(p)}$. In particular, for such a $g$, the equation $x^{p^i}=g$ has a solution if and only if the equations $x_j^{p^i}=g_j, j=1, 2$ have solutions. Moreover, in affirmative case, the former equation has $n_i(\mathcal C_{G_1}^{(p)})\cdot n_i(\mathcal C_{G_2}^{(p)})$ solutions in $\cup \mathcal C_{G_1\times G_2}^{(p)}$ since $G_1, G_2$ are $\alpha$-groups. For the primes dividing the order of only one $G_j$, the proof follows similarly.\end{proof}

\begin{definition}
For each $S\subseteq \C_G$, let $\mu_S$ denote the set of all elements of $S$ that are maximal (in $S$) with respect to the inclusion. 
\end{definition}

Before proceeding to the proof of Theorem~\ref{thm:main-alpha}, we need two technical results.

\begin{proposition}\label{betas}
    Let $G$ be a finite group and let $S\subseteq C_G^{(p)}$ be an $\alpha$-set. If $S\setminus \mu_S$ is nonempty, then it is also an $\alpha$-set. In this case, we have $n_{i}(S\setminus \mu_S)=\frac{n_{i+1}(S)}{n_1(S)}$ for every $i\ge 1$.
\end{proposition}
\begin{proof}
From hypothesis, $S$ is hereditary. In particular, we easily conclude that $S\setminus \mu_S$ is also hereditary.
Now observe that $h\in\cup (S\setminus\mu_S)$ if and only if there exists $y\in\cup S$ such that $y^p=h$.  In particular, for each $g\in \cup(S\setminus \mu_S)$, the number $N_g$ of solutions to $x^{p^{i}}=g$ with $x\in \cup(S\setminus \mu_S)$ equals the number of distinct values of $y^p$ with $y$ running over all the solutions (over $\cup S$) to the equation $x^{p^{i+1}}=g$. Since $S$ is an $\alpha$-set, this last equation has $0$ or $n_{i+1}(S)$ solutions. In the first case, we conclude that $N_g=0$. In the second case, we use again the fact that $S$ is an $\alpha$-set to conclude that the $n_{i+1}(S)$ solutions $y\in \cup S$ to $x^{p^{i+1}}=g$ can be grouped in subsets of size $n_1(S)$ according to the value of $y^p$. Hence $N_g=\frac{n_{i+1}(S)}{n_1(S)}$, concluding the proof.
\end{proof}

\begin{lemma}\label{lm3}
    Let $S\subseteq \C_G$ be an $\alpha$-set. If $C\in S\setminus \mu_S$, then the number of elements $C_0\in S$ with $C_0=p\cdot \# C$ and $C\subset C_0$ is equal to $\frac{n_1(S)}{p}$ if $C\neq \{1_G\}$ and equal to $\frac{n_1(S)-1}{p-1}$ if $C=\{1_G\}$.   
\end{lemma}
\begin{proof}
We first consider the case $C=\{1_G\}$. Since any two distinct cyclic groups of order $p$ intersect trivially, it follows that $S$ has exactly $\frac{n_1(S)-1}{p-1}$ groups of order $p=p\cdot \# \{1_G\}$. Suppose $C\ne \{1_G\}$, take $\theta\in G$ a generator for $C$ and let $p^k$ be the number of elements of $C$. Since $C\ne \{1_G\}$ and $C$ is not maximal, then $k\ge 1$ and $\theta$ is of the form $g^p$ with $g\in \cup S$. Since $S$ is an $\alpha$-set, the number of elements $x\in \cup S$ such that $x^p=\theta$ equals $n_1(S)$. Now take $g_1, g_2\in S$ with
$$        g_1^p=g_2^p=\theta.
$$ Since $k\ge 1$, it follows that $g_1$ and $g_2$ have order $p^{k+1}$. 
We note that $g_1$ and $g_2$ generate the same cyclic group if and only if there exists $1\le r\le p^{k+1}$ such that $\gcd(p, r)=1$ and \begin{equation}\label{eq2}
        g_1^r=g_2 
    \end{equation} Since $g_1^p=g_2^p=\theta$, Eq.~\eqref{eq2} holds if and only if $\theta^r=\theta$. Since $\theta$ has order $p^k$, the last equality holds if and only if $r\equiv1\mod p^{k}$, that is, $r=sp^k+1$ with $0\leq s\leq p-1$. Now observe that a cyclic subgroup $C_0$ of $S$ of order $p\cdot \# C$ contains $C$ if and only if there exists $g\in C_0$ such that $g^p=\theta$. Therefore, from our previous observation, the number of groups $C_0$ equals $n_1(S)/p$.
 \end{proof}

\begin{remark}\label{rmk1}
    The number of elements of order $p^i$ in an $\alpha$-set $S$ is equal to $n_1(S)$ if $i=1$ and equal to $n_i(S)-n_{i-1}(S)$ if $i>1$.
\end{remark}

\subsection{Proof of Theorem~\ref{thm:main-alpha}}

We start proving a more general result stating that $\la(S)\cong \la(T)$ whenever $S, T$ are $\alpha$-sets with the same order multiset. The proof is by induction and mainly relies on the fact that, for an $\alpha$-set $S$, the graph $\la(S)$ is well structured and can be easily obtained from the graph $\la(S\setminus \mu_S)$; thanks to Proposition~\ref{betas}, $S\setminus \mu_S$ is also an $\alpha$-set.

\begin{lemma}\label{lemabetas}
    Let $p$ be a prime and let $G,H$ be finite groups of order divisible by $p$. If $S\subseteq \C_G^{(p)}$ and $ T\subseteq \C_H^{(p)}$ are $\alpha$-sets such that $\ou(S)=\ou(T)$, then $\la(S)\cong\la(T)$.
\end{lemma}
\begin{proof}
Let $\exp(S)$ (resp. $\exp(T)$) be the largest $m\ge 0$ such that $p^m\in \ou(S)$ (resp. $p^m\in \ou(T)$). From hypothesis, $\exp(S)=\exp(T)$. So we prove our result by taking induction on $n:=\exp(S)$.
The case $n=0$ is straightforward since both $S, T$ comprise only the trivial group. For $n=1$, suppose $\exp(S)=1$ and $\ou(S)=\ou(T)$. In particular, both $S, T$ comprise $n_1(S)-1$ elements of order $p$ and $1$ element of order $1$. Therefore, $\la(S)$ and $\la(T)$ comprise a directed rooted tree with exactly $\frac{n_1(S)-1}{p-1}$ vertices of weight $p$, directed to the root which has weight $1$. In particular, $\la(S)\cong \la(H)$.
    
    Now fix $n\ge 1$ and assume the result holds whenever $\exp(S)=n$. Let $S, T$ be two $\alpha$-sets with $\exp(S)=n+1$ such that $\ou(S)=\ou(T)$. 
    Since $n\ge 1$, the sets  $\s=S\setminus\mu_S$ and $\T=T\setminus\mu_T$ are nonempty and $\exp(\s)=n$. From Proposition~\ref{betas}, $\s$ and $\T$ are $\alpha$-sets. Moreover, since $\ou(S)=\ou(T)$, we obtain $n_i(S)=n_i(T)$ for every $i\ge 1$. The latter combined with Proposition~\ref{betas} yields $n_i(\s)=n_i(\T)$ and then $\ou(\s)=\ou(\T)$ from Remark~\ref{rmk1}.
    
    As $\s, \T$ are $\alpha$-sets, the induction hypothesis entails that $\la(\s)\cong \la(\T)$. Let $f:\s\to \mathcal \T$ be the bijection inducing the isomorphism between these two digraphs.  For each $C\in \s$, let $M_S(C)$ be the set of all elements $D\in \mu_S$ such that $C\subsetneq D$ and there is no cyclic group $E$ with $C\subsetneq E\subsetneq D$ (that is, $C$ is a maximal subgroup of $D\in \mu_S$).
       
\begin{itemize}
     \item It follows by the definition that $M_S(C)$ is precisely the set of  $C_0\in \mu_S$ such that $\# C_0=p\cdot \# C$ and $C\subseteq C_0$.

      \item Since $\mu_S=S\setminus \s$ and both $S, \s$ are $\alpha$-sets, the previous observation combined with Lemma~\ref{lm3} provides

      \begin{equation}\label{ms1}
        \#M_S(\{1_G\})=\frac{n_1(S)-n_1(\s)}{p-1}
\end{equation} (since $n\ge 1$, $\{1_G\}$ is not in $\mu_\s$ or $\mu_S$) and
\begin{equation}\label{eq:magic}\# M_S(C)=\begin{cases}
    \frac{n_1(S)}{p}&\text{ if }C\in\mu_{\s},\\
    \frac{n_1(S)-n_1(\s)}{p} &\text{ if } C\in \s\setminus\mu_\s.
\end{cases}
\end{equation}

     \item From the first observation, any generator $g$ of $C_0\in M_S(C)$ is such that $g^p$ generates $C$. The latter implies that the sets $M_S(C)$ with $C\in \s$ are pairwise disjoint.
   
    \item From construction, $\mu_S=\bigcup_{C\in \s} M_S(C)$ and, for $C\in \s$, $D\in M_S(C)$ and $E\in S$, we have $E\subsetneq D$ if and only if $E\subseteq C$.

\end{itemize}

The observations above lead to the following conclusion: the graph $\la(S)$ is obtained from the graph $\la(\s)$ by adding, for each $C\in \s$, the set of vertices $M_S(C)$ (with weight $p\cdot \# C$) and directed edges $D\to E$ with $D\in M_S(C)$ and $E\in S$ such that $E=C$ or $C\to E$ is an edge in $\la(\s)$. We can argue similarly and obtain the same observations and conclusions for $\T$ and $\la(\T)$. 
With this notation in mind, we observe that $f(\mu_{\s})=\mu_{\T}$. In fact, $\mu_{\s}$ (resp. $\mu_{\T}$) corresponds to the set of vertices in $\la(\s)$ (resp. $\la(\T)$) that do not have edges directed to them. Moreover, since $f$ preserves the weight of the vertices, we obtain $\# f(C)=\# C$  for each $C\in \s$. 
In particular, from the previous description of the graphs $\la(S), \la(T)$ we see that they are isomorphic if, for each $C\in \s$, we have $\#M_S(C)=\#M_T(f(C))$. In fact, in this case, $f$ lifts to an isomorphism between $\la(S)$ and $\la(T)$ by considering, for each $C\in \s$, an arbitrary bijection between the sets $M_S(C)$ and $M_T(f(C))$. Since $f$ is a bijection with $f(\mu_{\s})=\mu_{\T}$, we obtain the equality $f(\s \setminus \mu_{\s})=\T\setminus \mu_{\T}$. Moreover, we clearly have $f(\{1_G\})=\{1_{H}\}$. In particular, considering Eqs.~\eqref{ms1} and~\eqref{eq:magic}, we obtain $\#M_S(C)=\#M_T(f(C))$ if we prove that $n_i(S)=n_i(T)$ and $n_i(\s)=n_i(\T)$ for every $i\ge 1$. The last two equalities hold thanks to Remark~\ref{rmk1} (recall that $\ou(S)=\ou(T)$ and $\ou(\s)=\ou(\T)$). 
\end{proof}

We now finish the proof of Theorem~\ref{thm:main-alpha}: from Theorem~\ref{thm:newmain}, it suffices to verify that $P_1\ne P_2$ and $P_6\le P_3$. Example~\ref{ex:alpha} shows that any group of prime exponent $p$ is an $\alpha$-group. In particular, by taking $p$ an odd prime, and considering the two non isomorphic groups of order $p^3$ and exponent $p$ we obtain that $P_1\ne P_2$ for the class of nilpotent $\alpha$-groups. Now let $G, H$ be nilpotent $\alpha$-groups with $G\sim_6 H$, that is, $\ou(G)=\ou(H)$. In particular, $\ou(\C_G^{(p)})=\ou(\C_H^{(p)})$ for every prime $p\mid\#G$. Since $G, H$ are $\alpha$-groups, it follows that $\C_G^{(p)}$ and $\C_H^{(p)}$ are $\alpha$-sets and then Lemma~\ref{lemabetas} implies $\la(\C_G^{(p)})=\la(\C_H^{(p)})$. Thus $\la(G)\cong \la(H)$ by Proposition~\ref{prop:prime} (recall that $G, H$ are nilpotent). In other words, $G\sim_3 H$.

\section{Conclusion and open problems}
In this paper, we introduced and explored several group invariants that arise from combinatorial objects related to finite groups, comparing the partitions they induce. Most notably, we proved that exactly three of them do not coincide: the isomorphism of groups, the isomorphism of power graphs, and the multiset $\ou(G)$.

In key parts of our proofs, we needed to extend the digraphs $\G$, $\G_t$, and $\la$ to collections of cyclic subgroups of $G$. In particular, the definitions given in Subsection~\ref{subsection:power-closed} can be used to extend the relations $\sim_j$ to arbitrary {\em power closed sets} (which are not groups in general). In this broader setting, we can extend 
the equivalence relations $\sim_j$ with $2\le j\le 6$ and it is not difficult to verify that our main theorems still hold. 

To this end, we propose some open problems. The first one seeks to decide whether Theorem~\ref{thm:newmain} holds for the class of all finite groups.

\begin{problem}
    Prove or disprove: there exist finite groups $G, H$ such that $G\sim_5 H$ but $G\not\sim_2 H$.
\end{problem}

Related to the latter, the following problem seeks to identify further relations that could lie strictly between those we have defined.




\begin{problem}\label{prob:1}
   Provide additional group invariants and determine their position relative to those presented in this paper. In other words, identify relations $\sim$ that generate a partition $P$ on the set of finite groups such that $P_1 < P < P_2$, $P_2<P<P_5$ or $P_5 < P < P_6$.
\end{problem}

It is likely that any new invariant will yield classes of finite groups with special properties, such as the $\alpha$-groups introduced in this paper. A  related question to Problem~\ref{prob:1} is the following.

\begin{problem}
Determine the pairs of non isomorphic finite groups $(G, H)$ with $G\sim_2H$. 
\end{problem}

It follows from the definition that any finite group of prime exponent is a nilpotent $\alpha$-group. In particular, by Theorem~\ref{thm:main-alpha}, any two non-isomorphic groups of prime exponent $p$ and of the same order are equivalent under $\sim_2$. In our research, we did not find any additional classes (apart from taking suitable products of groups) where this phenomenon occurs.

Finally, motivated by Theorem~\ref{thm:main-alpha}, we have the following problem.

\begin{problem}
    Provide further classes of groups for which Theorem~\ref{thm:main} can be strengthened.  
\end{problem}

A further investigation of the $\alpha$-groups, considering both algebraic and combinatorial aspects, would also be of interest. In particular, finite $p$-groups satisfying the conditions in Lemma~\ref{lem:alpha-ex} have been examined previously in~\cite{Xu}.

\section*{Acknowledgements}
We thank the anonymous reviewer for helpful suggestions that enhanced the presentation of this work. A. Fernandes was supported by CNPq Grant 163279/2021-7 and partially supported by FAPESP (Brazil) grant 2025/14775-1. L. Reis was supported by CAPES (Brazil) - Finance Code 001, CNPq Grant 310583/2025-0 and FAPEMIG  grant APQ-01712-23. S. Ribas was supported by FAPEMIG grants RED-00133-21 and APQ-01712-23. C. Qureshi, L. Reis and S. Ribas were supported by CNPq grant 420721/2025-8.


\end{document}